\documentclass[11pt]{article}
\usepackage{amssymb,amsmath,amsthm,graphicx}
\usepackage{subcaption}
\usepackage{float}
\usepackage{color,enumerate}
\usepackage{fullpage}
\usepackage{cite}

\newtheorem{theorem}{Theorem}[section]
\newtheorem{lemma}[theorem]{Lemma}
\newtheorem{corollary}[theorem]{Corollary}
\newtheorem{proposition}[theorem]{Proposition}

\newtheorem{observation}[theorem]{Observation}

\theoremstyle{definition}

\newcommand{\po}{$1$-p.o.}

\title{$1$-perfectly orientable graphs and graph products \thanks{This work was supported in part by the Slovenian Research Agency (I$0$-$0035$, research programs P$1$-$0285$, research projects N$1$-$0032$, J$1$-$5433$, J$1$-$6720$, J$1$-$6743$, J$1$-$7051$, and a Young Researchers Grant).}}
\date{\today}

\author{
Tatiana Romina Hartinger\\
\small University of Primorska, UP IAM, Muzejski trg 2, SI6000 Koper, Slovenia\\
\small University of Primorska, UP FAMNIT, Glagolja\v ska 8, SI6000 Koper, Slovenia\\
\small \texttt{tatiana.hartinger@iam.upr.si}
\and
Martin Milani\v c\\
\small University of Primorska, UP IAM, Muzejski trg 2, SI6000 Koper, Slovenia\\
\small University of Primorska, UP FAMNIT, Glagolja\v ska 8, SI6000 Koper, Slovenia\\
\small \texttt{martin.milanic@upr.si}
}

\usepackage{amssymb}
\usepackage{graphicx}

\begin{document}
\maketitle

\begin{abstract}
A graph $G$ is said to be \emph{1-perfectly orientable} (\po~for short) if it admits an orientation such that the out-neighborhood of every
vertex is a clique in $G$. The class of \po~graphs forms a common generalization of the classes of chordal and circular arc graphs.
Even though \hbox{\po~graphs} can be recognized in polynomial time, no structural characterization of \po~graphs is known.
In this paper we consider the four standard graph products: the Cartesian product, the strong product, the direct product, and the lexicographic
product. For each of them, we characterize when a nontrivial product of two graphs is \po
\end{abstract}


\section{Introduction}

A \emph{tournament} is an orientation of a complete graph. We study graphs having an orientation that is an {\it out-tournament}, that is, a
digraph in which the out-neighborhood of every vertex induces a \emph{tournament}. ({\it In-tournaments} are defined similarly.) Following
the terminology of Kammer and Tholey~\cite{MR3152051}, we say that an orientation of a graph is {\it $1$-perfect} if it is an out-tournament,
and that a graph is {\it $1$-perfectly orientable} ({\it $1$-p.o.}~for short) if it has  a $1$-perfect orientation. In~\cite{MR3152051}, Kammer and
Tholey introduced the more general concept of \emph{$k$-perfectly orientable} graphs, as graphs admitting an orientation in which the
out-neighborhood of each vertex can be partitioned into at most $k$ sets each inducing a tournament. They developed several approximation algorithms for
optimization problems on $k$-perfectly orientable graphs and related classes. It is easy to see (simply by reversing the arcs) that \po~graphs are exactly the graphs that admit an orientation that is an in-tournament. In-tournament orientations were called {\it fraternal orientations} in several papers~\cite{MR1246675,MR1449722,MR1292980,MR1287025,MR2323998,
MR2548660, MR1161986}.

The concept of \po~graphs was introduced in 1982 by Skrien~\cite{MR666799} (under the name {\it $\{B_2\}$-graphs}), where the problem of characterizing
\po~graphs was posed. While a structural understanding of \po~graphs is still an open question, partial results are known.
Bang-Jensen et al.~observed in~\cite{MR1244934} that \po~graphs can be recognized in polynomial time via a reduction to $2$-SAT. Skrien~\cite{MR666799}
characterized graphs admitting an orientation that is both an in-tournament and an out-tournament as exactly the proper circular arc graphs. All chordal
graphs and all circular arc graphs are $1$-p.o.~\cite{MR1161986}, and, more generally, so is any vertex-intersection graph of connected induced subgraphs of a unicyclic graph~\cite{MR1244934,Prisner-thesis}. Every graph having a unique induced cycle of order at least $4$~is
$1$-p.o.~\cite{MR1244934}.

In~\cite{HM2016}, several operations preserving the class of \po~graphs were described (see Sec.~\ref{sec:prelim}); operations that do not preserve the property in general were also considered. In the same paper \po~graphs were characterized in terms of edge-clique covers, and
characterizations of \po~cographs and of \po~co-bipartite graphs were given. In particular, a cograph is \po~if and only if it is $K_{2,3}$-free and a co-bipartite graph is \po~if and only if it is circular arc. A structural characterization of line graphs that are~\po~was given in~\cite{MR1244934}.

In this paper we consider the four standard graph products: the Cartesian product, the strong product, the direct product, and the lexicographic
product. For each of these four products, we completely characterize when a nontrivial product of two graphs $G$ and $H$ is $1$-p.o.
While the results for the Cartesian, the lexicographic, and the direct products turn out to be rather straightforward,
the characterization for the case of the strong product is more involved.

Some common features of the structure of the factors involved in the characterizations can be described as follows.
In the cases of the Cartesian and the direct product the factors turn out to be very sparse and very restricted, always
having components with at most one cycle. In the case of the lexicographic and of the strong product
the factors can be dense. More specifically, co-bipartite \hbox{\po~graphs}, including co-chain graphs in the case of strong products, play an important role in these characterizations. The case of the strong product also leads to a new infinite family of \po~graphs (cf.~Proposition~\ref{lemm:orient-rafts}).

The paper is organised as follows. Section 2 includes the basic definitions and notation, and recalls several known results about \po~graphs that will be required
for some of the proofs. In Sections 3, 4, 5, and 6 we deal, respectively, with \po~Cartesian product graphs, \po~lexicographic product graphs, \po~direct product
graphs, and \po~strong product graphs, and state and prove the corresponding characterizations.


\section{Preliminaries}\label{sec:prelim}

All graphs considered in this paper are simple and finite, but may be either undirected or directed (in which case we refer to them as digraphs).
An edge in a graph connecting vertices $u$ and $v$ will be denoted simply $uv$. The {\it neighborhood} of a vertex $v$ in a graph $G$
is the set of all vertices adjacent to $v$ and will be denoted by $N_{G}(v)$. The {\it degree} of $v$ is the size of its neighborhood. A {\it leaf} in a graph is a vertex of degree~$1$. The {\it closed neighborhood} of $v$ in $G$ is the set $N_G(v)\cup \{v\}$, denoted by $N_G[v]$.
An \emph{orientation} of a graph $G = (V, E)$ is a digraph $D= (V,A)$ obtained by assigning  a direction to each edge of $G$.
Given a digraph $D=(V,A)$, the \emph{in-neighborhood} of a vertex $v$ in $D$, denoted by $N^{-}_{D}(v)$, is the set of all vertices $w$ such
that
$(w,v) \in A$. Similarly, the \emph{out-neighborhood} of $v$ in $D$ is the set $N^{+}_{D}(v)$ of all vertices $w$ such that
$(v,w) \in A$. We may omit the subscripts when the corresponding graph or digraph is clear from the context.
Given an undirected graph $G$ and a set $S\subseteq V(G)$, we define the {\it neighborhood of $S$} as
$N(S)= (\bigcup_{x \in S}{N(x)}) \setminus S$. The {\it subgraph of $G$ induced by $S$} is the graph, denoted by $G[S]$,
with vertex set $S$ and edge set $\{uv: u\in S, v\in S, uv\in E(G)\}$. The {\it distance} between two vertices $x$ and $y$ in a connected graph $G$
will be denoted by $d_G(x,y)$ (or simply $d(x,y)$) and defined, as usual, as the length of a shortest $x$-$y$ path.

Given two graphs $G$ and $H$, their {\it union} is the graph $G\cup H$ with vertex set $V(G)\cup V(H)$ and
edge set $E(G)\cup E(H)$. Their {\it disjoint union} is the graph $G+H$ with vertex set $V(G)\;\dot{\cup}\;V(H)$ (disjoint union) and
edge set $E(G)\cup E(H)$ (if $G$ and $H$ are not vertex disjoint, we first replace one of them with a disjoint isomorphic copy). We write $2G$ for $G+G$. The {\it join} of two graphs $G$ and $H$ is the graph denoted by $G\ast H$ and obtained from the
disjoint union of $G$ and $H$ by adding to it all edges joining a vertex of $G$ with a vertex of $H$. Given two graphs $G$ and $H$ and a vertex
$v$ of $G$, the \emph{substitution} of $v$ in $G$ for $H$ consists in replacing $v$ with $H$ and making each vertex of $H$ adjacent to every
vertex in $N_{G}(v)$ in the new graph.

A {\it clique} (resp.,~{\it independent set}) in a graph $G$ is a set of pairwise adjacent (resp.,~non-adjacent) vertices of $G$.
The {\it complement} of a graph $G$ is the graph $\overline{G}$ with the same vertex set as $G$ in which two distinct vertices are adjacent
if and only if they are not adjacent in $G$. The fact that two graphs $G$ and $H$ are isomorphic to each other will be denoted by
$G\cong H$. Given a family ${\cal F}$ of graphs, we say that a graph is \emph{${\cal F}$-free} if it has no induced subgraph isomorphic to a
graph of ${\cal F}$.

$K_n$, $C_n$ and $P_n$ denote the $n$-vertex complete graph, cycle, and path, respectively. The \emph{claw} is the complete bipartite graph
$K_{1,3}$, that is, a star with $3$ edges, $3$ leaves and one central vertex. The \emph{bull} is a graph with $5$ vertices and $5$ edges, consisting of a triangle
with two disjoint pendant edges. The \emph{gem} is the graph $P_4 \ast K_1$, that is, the $5$-vertex graph consisting of a $4$-vertex path plus a vertex
adjacent
to each vertex of the path.

For graph theoretic notions not defined above, see, e.g.~\cite{opac-b1096791}. We will recall the definitions and some basic facts about each of the four graph products studied in the respective sections (Sec.~$3$--$7$). For each of the four considered products, we say that the product
of two graphs is \emph{nontrivial} if both factors have at least $2$ vertices.
For further details regarding product graphs and their properties, we refer to~\cite{opac-b1132990,MR1788124}.

In~\cite{HM2016}, several results about \po~graphs were proved. In the rest of this section we list some of them for later use.

\begin{proposition}\label{prop:non-1-po}
No graph in the set $\{F_1,F_2,F_3,F_4\}$ (see Figure~\ref{fig:1}) is $1$-perfectly orientable.
\end{proposition}

 \begin{figure}[h!]
  \centering
\includegraphics[width=0.75\textwidth]{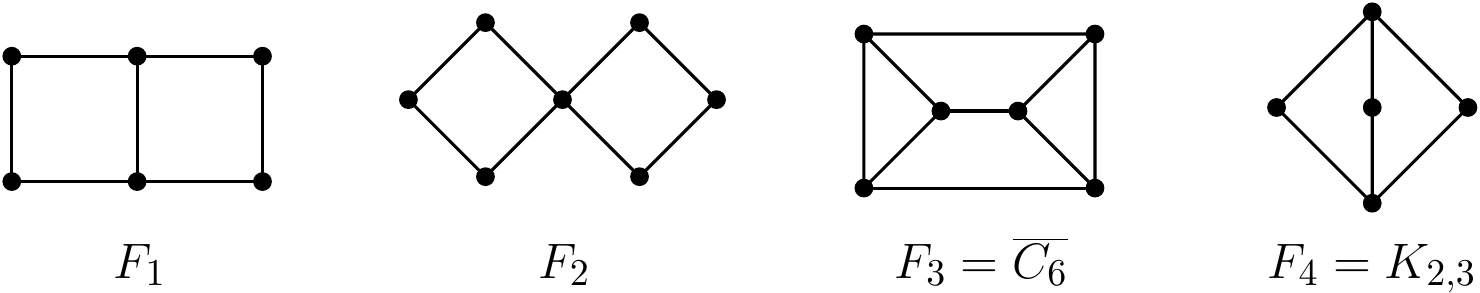}
\caption{Four small non-\po~graphs.}\label{fig:1}
\end{figure}

Two distinct vertices $u$ and $v$ in a graph $G$ are said to be \emph{true twins} if $N_{G}[u] = N_{G}[v]$.
We say that a vertex $v$ in a graph $G$ is \emph{simplicial} if its neighborhood forms a clique and \emph{universal} if it is adjacent to all other
vertices of the graph, that is, if $N_{G}[v]=V(G)$.
The operations of adding a true twin, a universal vertex, or a simplicial vertex to a given graph are defined in the obvious way.

\begin{proposition}\label{prop:operations}
The class of \po~graphs is closed under each of the following operations:
\begin{enumerate}[(a)]
  \item Disjoint union.
  \item Adding a true twin.
  \item Adding a universal vertex.
  \item Adding a simplicial vertex.
\end{enumerate}
\end{proposition}

Recall that a graph $H$ is said to be an {\it induced minor} of a graph $G$ if $H$ can be obtained from $G$ by a (possibly empty) sequence of
vertex deletions and edge contractions, where {\it contracting an edge} $uv$ in a graph $G$ means deleting its endpoints and adding a new vertex
adjacent exactly to vertices in $N_{G}(u) \cup N_{G}(v)$.

\begin{proposition}\label{prop:minors}
If $G$ is \po~and $H$ is an induced minor of $G$, then $H$ is \po
\end{proposition}

Propositions~\ref{prop:operations} and~\ref{prop:minors} imply the following.

\begin{corollary}\label{cor:connected}
A graph $G$ is \po~if and only if each component of $G$ is \po
\end{corollary}

And from Propositions~\ref{prop:non-1-po} and~\ref{prop:minors} we obtain the following.

\begin{corollary}\label{cor:minors}
Let $G$ be a graph such that some graph $F_i$ (with $1\le i\le 4$, see Figure~\ref{fig:1}) is an induced minor of it. Then $G$ is not \po
\end{corollary}

Since the class of \po~graphs is closed under induced minors, it can be characterized in terms of {\it minimal forbidden induced minors}. That is,
there exists a unique minimal set of graphs $\tilde{\cal F}$ such that a graph $G$ is \po~if and only if
$G$ is {\it $\tilde{\cal F}$-induced-minor-free}. Such a set is minimal in the sense that every induced minor from a graph in {\it $\tilde{\cal F}$}
is \po~In a recent paper~\cite{HM2016}, ten specific minimal forbidden induced minors
for the set of \po~graphs, along with two infinite families (generalizing the graphs $F_3$ and $F_4$ from Figure~\ref{fig:1}, respectively)
were identified. However, a complete set $\tilde{\cal F}$ of minimal forbidden induced minors is unknown.

A \emph{bipartite graph} is a graph whose vertex set can be partitioned into two independent sets. A graph is said to be \emph{co-bipartite} if
its complement is bipartite. Co-bipartite \po~graphs play an important role in the characterization of when the join of two graphs is $1$-p.o.

\begin{theorem}\label{prop:co-bipartite-join}
For every two graphs $G_1$ and $G_2$, their join $G_1\ast G_2$ is \po~if and only if one of the following conditions holds:
\begin{enumerate}[(i)]
  \item $G_1$ is a complete graph and $G_2$ is a \po~graph, or vice versa.
  \item Each of $G_1$ and $G_2$ is a co-bipartite \po~graph.
\end{enumerate}
In particular, the class of co-bipartite \po~graphs is closed under join.
\end{theorem}


\section{\po~Cartesian product graphs}

We start with a characterization of nontrivial Cartesian product graphs that are $1$-p.o.
The \emph{Cartesian product} $G \Box H$ of two graphs $G$ and $H$ is the graph with vertex set
$V(G) \times V(H)$ in which two distinct vertices $(u,v)$ and $(u',v')$ are adjacent if and only if
\begin{enumerate}[(a)]
\item $u = u'$ and $v$ is adjacent to $v'$ in $H$, or
\item $v = v'$ and $u$ is adjacent to $u'$ in $G$.
\end{enumerate}

The Cartesian product of two graphs is commutative, in the sense that $G \Box H \cong H \Box G$.
The following theorem characterizes when a nontrivial Cartesian product graph is $1$-p.o. For the proof, let us note that $P_3 \Box K_2$ is isomorphic to the {\it domino} (graph $F_1$ in Fig.~\ref{fig:1}), and $K_3 \Box K_2$ is isomorphic to $\overline{C_6}$ (the first graph in the family ${\cal F}_1$, see Fig.~\ref{fig:1}).

\begin{theorem}\label{thm:cartesian}
A nontrivial Cartesian product, $G \Box H$, of two graphs $G$ and $H$ is \po~if and only if
one of the following conditions holds:
\begin{enumerate}[(i)]
  \item $G$ is edgeless and $H$ is \po, or vice versa.
  \item $G\cong pK_1+qK_2$ and $H\cong rK_1+sK_2$ for some $p,q,r,s\ge 0$.
 \end{enumerate}
\end{theorem}

\begin{proof}
Suppose first that $G \Box H$ is $1$-p.o., and, for the sake of contradiction, that none of the conditions $(i)$ and $(ii)$ hold.
Since both $G$ and $H$ are induced subgraphs of $G\Box H$, they are both \po~(by Corollary~\ref{cor:minors}).
Since $(i)$ does not hold, each of $G$ and $H$ contains an edge.
Since property $(ii)$ does not hold, we may assume that $G$ contains a component with at least three vertices.
In particular, $G$ contains $P_3$ as a (not necessarily induced) subgraph.
If $G$ contains an induced $P_3$, then $G\Box H$ contains an induced domino, and is therefore not \po~by Corollary~\ref{cor:minors}.
Similarly, if $G$ contains an induced $K_3$, then $G\Box H$ contains an induced copy of $K_3 \Box K_2\cong \overline{C_6}$, and is therefore not \po, again by
Corollary~\ref{cor:minors}.
In either case, we reach a contradiction.

Conversely, suppose that one of the conditions $(i)$ and $(ii)$ holds.
If condition $(i)$ holds, we may assume that $G$ is edgeless and $H$ is \po, then
the product $G \Box H$ is isomorphic to a disjoint union of $|V(G)|$ copies of $H$, and thus \po~by
Proposition~\ref{prop:operations}.
If condition $(ii)$ holds, then each component of the product $G \Box H$ is isomorphic to either
$K_1$, $K_2$, or $C_4$, which are all \po~graphs (the cyclic orientation of $C_4$ is $1$-perfect). To obtain the desired conclusion, we again apply
the fact that \po~graphs are closed under disjoint union.
\end{proof}


\section{\po~lexicographic product graphs}

In this section, we characterize nontrivial lexicographic product graphs that are $1$-p.o.
Given two graphs $G$ and $H$, the \emph{lexicographic product} of $G$ and $H$, denoted by $G[H]$ (sometimes also by $G \circ{H}$) is the graph with vertex
set $V(G) \times{V(H)}$, in which two distinct vertices $(u,v)$ and $(u',v')$ are adjacent if and only if
\begin{enumerate}[(a)]
\item $u$ is adjacent to $u'$ in $G$, or
\item $u = u'$ and $v$ is adjacent to $v'$ in $H$.
\end{enumerate}

Note that contrary to the other three products considered in this paper, the lexicographic product is not commutative, that is, $G[H] \ncong H[G]$
in general.
The following theorem characterizes when a nontrivial lexicographic product graph is \po

\begin{theorem}\label{thm:lexicographic}
A nontrivial lexicographic product, $G[H]$, of two graphs $G$ and $H$ is \po~if and only if one of the following conditions holds:
\begin{enumerate}[(i)]
\item $G$ is edgeless and $H$ is \po
\item $G$ is \po~and $H$ is complete.
\item Every component of $G$ is complete and $H$ is a co-bipartite \po~graph.
\end{enumerate}
\end{theorem}

\begin{proof}
Suppose first that $G[H]$ is $1$-p.o.~Then, both $G$ and $H$ are \po~since they are induced subgraphs of $G[H]$.
Suppose for the sake of contradiction that none of conditions $(i)$--$(iii)$ holds. Then, in particular, $G$ has an edge and $H$ is not complete.
Since $G$ has an edge, we get that $K_2[H]$ is an induced subgraph of $G[H]$ isomorphic to the join
of two copies of $H$. Consequently, $H \ast H$ is \po~By Proposition~\ref{prop:co-bipartite-join} we obtain that $H$ is co-bipartite.
Therefore, since we assume that $(iii)$ fails, $G$ has a component that is not complete. In particular, there exists an induced $P_3$ in $G$; since $H$ contains
an
induced $2K_1$ and $P_3[2K_1] \cong K_{2,4}$ we obtain that $G[H]$ contains $K_{2,3}$ as an induced subgraph, and by Corollary~\ref{cor:minors} it cannot
be
\po

For the converse direction, we will show that in any of the three cases $(i)$, $(ii)$, and $(iii)$, the graph $G[H]$ is $1$-p.o.
If $G$ is edgeless and $H$ is \po, then the product $G[H]$ is isomorphic to the disjoint union of $|V(G)|$ copies of $H$, and therefore \po~by
Proposition~\ref{prop:operations}.
If $G$ is \po~and $H$ is complete, then the product $G[H]$ is isomorphic to the graph obtained by repeatedly substituting a vertex of $G$
with a complete graph. Substituting a vertex $v$ with a complete graph is the same as adding a sequence of true twins to vertex $v$, which by
Proposition~\ref{prop:operations} results in a \po~graph. It follows that $G[H]$ is $1$-p.o.
Finally, suppose that every component of $G$ is complete and $H$ is a co-bipartite \po~graph.
Since if $G$ has components $G_1,\ldots, G_k$, then the components of
$G[H]$ are $G_1[H],\ldots, G_k[H]$ and the set of \po~graphs is closed under disjoint union, it suffices to consider the case
when $G$ is connected (that is, complete). In this case, an inductive argument on the order of $G$ together with the fact that
co-bipartite \po~graphs are closed under join (Proposition~\ref{prop:co-bipartite-join}) shows that $G[H]$ is \po
\end{proof}

\section{\po~direct product graphs}

In this section, we characterize nontrivial direct product graphs that are $1$-p.o.
The \emph{direct product} $G \times H$ of two graphs $G$ and $H$
(sometimes also called {\it tensor product}, {\it categorical product}, or {\it Kronecker product})
is the graph with vertex set
$V(G) \times V(H)$ in which two distinct vertices $(u,v)$ and $(u',v')$ are adjacent if and only if
\begin{enumerate}[(a)]
\item $u$ is adjacent to $u'$ in $G$, and
\item $v$ is adjacent to $v'$ in $H$.
\end{enumerate}

The direct product of two graphs is commutative, in the sense that $G \times H \cong H \times G$. By~\cite[Corollary 5.10]{opac-b1132990}, the direct
product of (at least two) connected nontrivial graphs is connected if and only if at most one of the factors is bipartite
(in fact, the product has $2^{k-1}$ components where $k$ is the number of bipartite factors).

We start with some necessary conditions for the direct product of two graphs to be $1$-p.o.
We say that a graph is {\it triangle-free} if it is $C_3$-free.

\begin{lemma}\label{lem:forbidden-direct}
Suppose that the direct product of two graphs $G$ and $H$ is $1$-p.o. Then:
\begin{enumerate}
\item If one of $G$ and $H$ contains an induced $P_3$ or $C_3$, then the other one is $\{{\it claw}, C_3, C_4, C_5, P_5\}$-free.
\item At least one of $G$ and $H$ is triangle-free.
\item At least one of $G$ and $H$ is $P_4$-free.
\end{enumerate}
\end{lemma}

\begin{proof}
%
%
%

As we can see in Figures~\ref{fig:5},~\ref{fig:6}, and~\ref{fig:7} below, each of
$P_3 \times {\it claw}$, $P_3 \times C_4$, and $C_3 \times {\it claw}$
contains an induced $K_{2,3}$,
each of $P_3 \times C_3$, $P_3 \times C_5$, and $P_3 \times P_5$
contains an induced $F_2$,
the graph $C_3 \times C_3$ contains an induced $F_3 = \overline{C_6}$,
and $P_4 \times P_4$ contains an induced domino ($F_1$).

\begin{figure}[H]
  \centering
\includegraphics[width=0.75\textwidth]{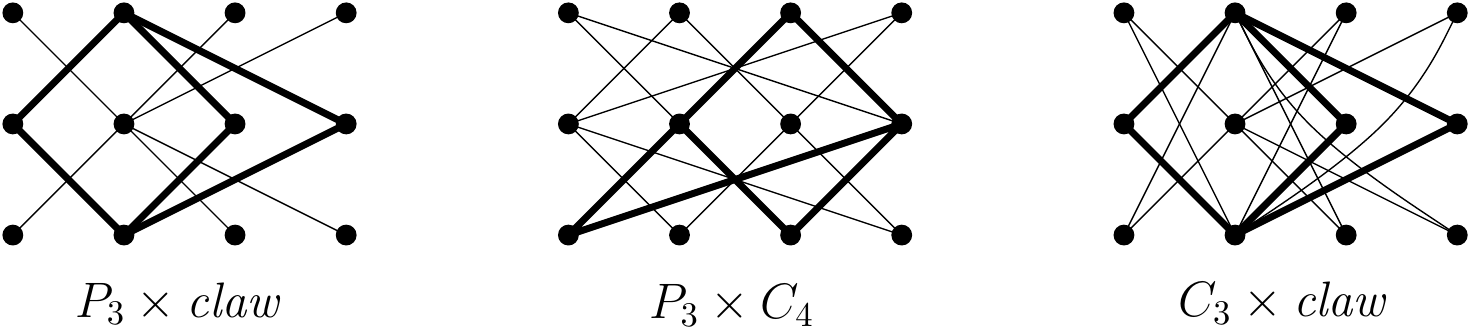}
\caption{$K_{2,3}$ as induced subgraph of $P_3 \times {\it claw}$, $P_3 \times C_4$, and $C_3 \times {\it claw}$.}\label{fig:5}
\end{figure}

\begin{figure}[H]
  \centering
\includegraphics[width=0.75\textwidth]{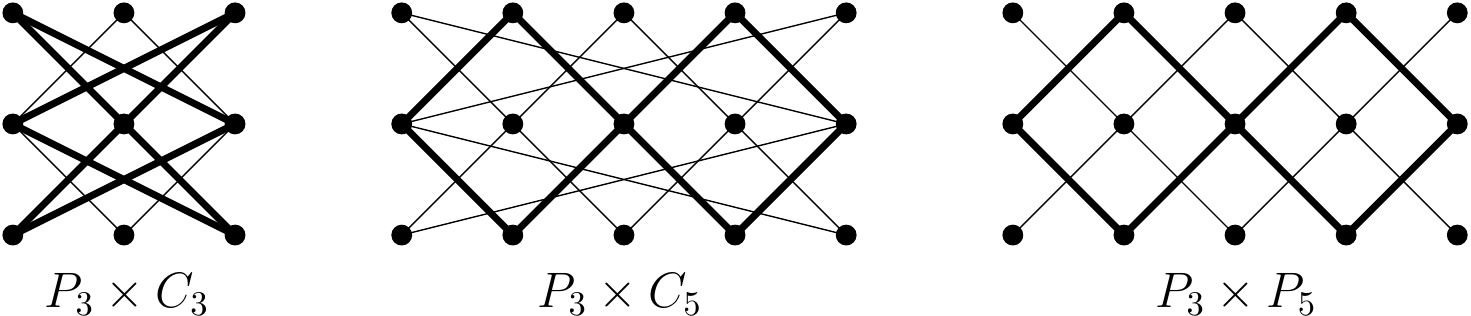}
\caption{$F_2$ as induced subgraph of $P_3 \times C_3$, $P_3 \times C_5$, and $P_3 \times P_5$.}\label{fig:6}
\end{figure}

\begin{figure}[H]
  \centering
\includegraphics[width=0.5\textwidth]{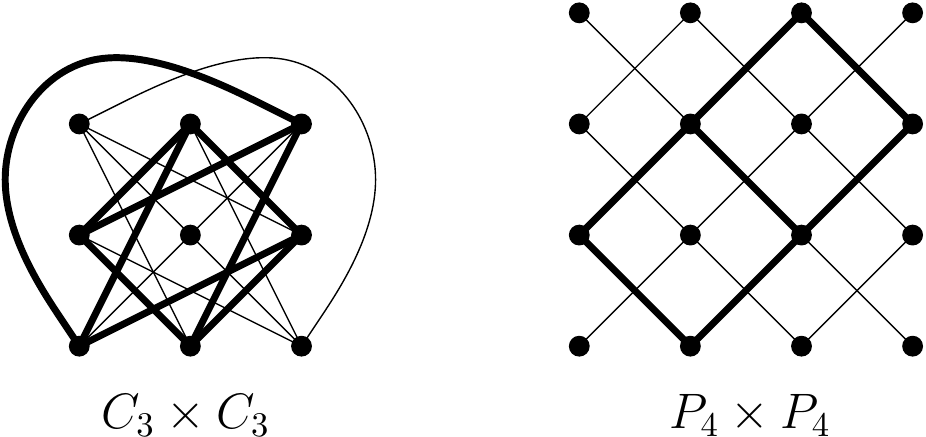}
\caption{The complement of $C_6$ as induced subgraph of $C_3 \times C_3$ and the domino as induced subgraph of $P_4 \times P_4$.}\label{fig:7}
\end{figure}

Each of $C_3\times C_4$, $C_3\times C_5$, and $C_3\times P_5$ contains an induced
$C_3 \times P_3\cong P_3 \times C_3$, and therefore also an induced $F_2$.

The lemma now follows from the above observations and Corollary~\ref{cor:minors}.
\end{proof}

We say that an undirected graph is a \emph{pseudoforest} if each component of it contains at most one cycle,
a \emph{pseudotree} if it is  a connected pseudoforest, and a \emph{unicyclic} graph if it contains exactly one cycle.
We first characterize the case of two connected factors.

\begin{proposition}\label{prop:direct}
A nontrivial direct product, $G \times H$, of two connected graphs $G$ and $H$ is \po~if and only if one of the following conditions holds:
\begin{enumerate}[(i)]
\item $G\cong K_2$ and $H$ is a pseudotree, or vice versa.
\item $G\cong P_3$ and $H\cong P_4$, or vice versa.
\item $G\cong H\cong P_3$.
\end{enumerate}
\end{proposition}

\begin{proof}
We first show that any of the three conditions $(i)$, $(ii)$, and $(iii)$ is sufficient for $G \times H$ to be $1$-p.o.
Recall that every chordal graph and every graph having a unique induced cycle of order at least $4$~is $1$-p.o.~\cite{MR1244934}.
In particular, this implies that every pseudoforest is $1$-p.o.
Suppose first that $G\cong K_2$ and $H$ is a pseudotree.
If $H$ is bipartite, then $K_2 \times H$ is isomorphic to the pseudoforest $2H$, which is $1$-p.o.
If $H$ is non-bipartite, then it is unicyclic, in which case $K_2 \times H$ is again unicyclic and therefore $1$-p.o.
Finally, $P_3\times P_4$ is isomorphic to $2F$ where $F$ is a unicyclic graph, and is therefore a \po~graph. This also implies that
$P_3\times P_3$ is \po

To show necessity, suppose that $G \times H$ is $1$-p.o. We consider two cases depending on whether
one of $G$ and $H$ is isomorphic to $K_2$ or not.
Suppose first that one of $G$ and $H$, say $G$, is isomorphic to $K_2$. Then $K_2 \times H$ is
triangle-free,
and it follows from~\cite[Corollary 5.7]{MR1244934} that $K_2 \times H$ is a pseudoforest.
If $H$ is bipartite, then the product $K_2 \times H$ is isomorphic to $2H$, therefore $H$ is a connected \po~bipartite graph, and by~\cite[Corollary
5.7]{MR1244934} $H$ must be a pseudotree.
Suppose now that $H$ is non-bipartite. Then, $K_2 \times H$ is connected~\cite[Theorem 5.9]{opac-b1132990} and hence a pseudotree.
Let us observe that in this case $H$ must be a unicyclic graph (and hence a pseudotree). Indeed, if $H$ has a cycle $(v_1, \ldots, v_k)$ (for some odd $k$) then
$K_2\times H$ has a cycle of length $2k$ formed by vertices $(u_1,v_1), (u_2, v_2), (u_1, v_3), \ldots, (u_1, v_k), (u_2,v_1),
(u_1,v_2), (u_2, v_3), \ldots, (u_2,v_k)$, where $u_1$ and $u_2$ are the two vertices of the $K_2$. Therefore if $H$ had more than one
cycle, then so would $K_2  \times H$, and we know that this is not the case.

Now consider the case when both $G$ and $H$ have at least $3$ vertices. By Lemma~\ref{lem:forbidden-direct}, at least one
of $G$ and $H$, say $G$, is triangle-free. Since $G$ has at least $3$ vertices, it contains an induced $P_3$.
Applying Lemma~\ref{lem:forbidden-direct} further, we infer that $H$ is $\{{\it claw}, C_3, C_4, C_5, P_5\}$-free.
Since $H$ is $\{{\it claw}, C_3\}$-free, it is of maximum degree at most $2$, thus a path or a cycle. Since $H$ is also
$\{C_4,C_5,P_5\}$-free and connected, we conclude that $H$ is a path with either $3$ or $4$ vertices.
If $H \cong P_4$, then $G$ is $P_4$-free by Lemma~\ref{lem:forbidden-direct}, and since it contains a $P_3$, we must have $G \cong P_3$. If $H\cong P_3$,
then applying the same arguments as above we obtain that $G \cong P_3$ or $G \cong P_4$.
This concludes the proof of the forward implication, and with it the proof of the theorem.
\end{proof}

We now characterize the general case.
To describe the result, the following notion will be convenient.
For a positive integer $k$, we say that a {\it $k$-linear forest} is a disjoint union of paths each having at most $k$ vertices.
In particular, $1$-linear forest are exactly the edgeless graphs, and $2$-linear forests are exactly the graphs consisting
only of isolated vertices and isolated edges.

\begin{theorem}\label{thm:direct}
A nontrivial direct product, $G \times H$, of two graphs $G$ and $H$ is \po~if and only if one of the following conditions holds:
\begin{enumerate}[(i)]
\item $G$ is a $1$-linear forest and $H$ is any graph, or vice versa.
\item $G$ is a $2$-linear forest and $H$ is a pseudoforest, or vice versa.
\item $G$ is a $3$-linear forest and $H$ is a $4$-linear forest, or vice versa.
\end{enumerate}
\end{theorem}

\begin{proof}
Suppose first that $G \times H$ is $1$-p.o.
If at least one of $G$ and $H$ is edgeless, then condition $(i)$ holds. Assume now that both $G$ and $H$ contain an edge.
We claim that every component $C$ of $G$ is a pseudotree (and by symmetry, the same conclusion will hold for components of $H$). This is
a consequence of Proposition~\ref{prop:direct}, using
the fact that $C\times K_2$ is an induced subgraph of $G \times H$ (and hence \po).
Thus, if $G$ is a $2$-linear forest, then condition $(ii)$ holds (and similarly for $H$).
Assume now that both $G$ and $H$ have a component with at least three vertices.
Fixing two such components, say $C$ and $D$, of $G$ and $H$, respectively, and applying Proposition~\ref{prop:direct}, we infer that
each of $C$ and $D$ is a path of order $3$ or $4$, and not both can be isomorphic to $P_4$.
Consequently, $G$ and $H$ are of the form specified in condition $(iii)$.

For the converse direction, we will show that in any of the three cases $(i)$, $(ii)$, and $(iii)$, the graph $G\times H$ is $1$-p.o.
If condition $(i)$ holds, then $G\times H$ is edgeless and thus $1$-p.o.
Suppose that
condition $(ii)$ holds, say $G$ is a $2$-linear forest and $H$ is a pseudoforest.
Since $G\times H$ is the disjoint union of graphs of the form $C\times D$ where $C$ and $D$ are components of $G$ and $H$, respectively,
and \po~graphs are closed under disjoint union (Proposition~\ref{prop:operations}), it suffices to show that for every such pair $C$ and $D$,
the graph $C\times D$ is $1$-p.o. If $C\cong K_1$ then $C\times D$ is edgeless, hence $1$-p.o., while if $C\cong K_2$ then
$C\times D$ is a product of $K_2$ with a pseudotree, and hence \po~by Proposition~\ref{prop:direct}.
Finally, if condition $(iii)$ holds, a similar approach shows that $C\times D$ is either edgeless,
a product of $K_2$ with a pseudotree, or isomorphic to either $P_3\times P_3$ or $P_3\times P_4$, hence in either case \po~by Proposition~\ref{prop:direct}.
It
follows that $G\times H$ is \po~as well.
\end{proof}

\section{\po~strong product graphs}

In this section, we characterize nontrivial strong product graphs that are $1$-p.o.
The \emph{strong product} $G \boxtimes H$ of graphs $G$ and $H$ is the graph with vertex set $V(G) \times V(H)$ in which two distinct
vertices $(u,v)$ and $(u',v')$ are adjacent if and only if
\begin{enumerate}[(a)]
\item $u$ is adjacent to $u'$ in $G$ and $v=v'$, or
\item $u=u'$ and $v$ is adjacent to $v'$ in $H$, or
\item $u$ is adjacent to $u'$ in $G$ and $v$ is adjacent to $v'$ in $H$.
\end{enumerate}
It is easy to see that the fact that one of the conditions (a), (b) and (c) holds is equivalent to the pair of conditions
$u' \in N_{G}[u] $ and $v' \in N_{H}[v]$, that is, that $(u',v')\in N_{G}[u]\times N_{H}[v]$.
Consequently, for every two vertices $u \in V(G)$ and $v \in V(H)$, we have $N_{G \boxtimes H}[(u,v)]=N_{G}[u] \times N_{H}[v]$.

The strong product of two graphs is commutative, in the sense that $G \boxtimes H \cong H \boxtimes G$.

\begin{sloppypar}
Our characterization of \po~strong product graphs will be proved in several steps.
In Section~\ref{subsec:prelim}, we state two preliminary lemmas on the strong product
and give two necessary conditions for \po~strong product graphs.
The necessary conditions motivate the development of a structural characterization of
$\{P_5, C_4, C_5, {\it claw}, {\it bull}\}$-free graphs.
This is done in Section~\ref{subsec:co-chain},
where connected $\{P_5, C_4, C_5, {\it claw}, {\it bull}\}$-free graphs
are shown to be precisely the connected co-chain graphs.
Connected true-twin-free co-chain graphs are further characterized in
Section~\ref{subsec:rafts}, and form the basis of an infinite family of \po~strong product
graphs described in Section~\ref{subsec:sufficient}.
Building on these results, we prove our main result of the section,
Theorem~\ref{thm:strong-general} in Section~\ref{subsec:characterization}, which gives a complete
characterization of \po~strong product graphs.
\end{sloppypar}


\subsection{Three lemmas}\label{subsec:prelim}

Recall that a vertex $v$ in a graph $G$ is \emph{simplicial} if its neighborhood forms a clique.
In Section~\ref{subsec:sufficient} we  will need the following property of simplicial vertices in relation to the strong product.

\begin{lemma}\label{lem:simplicial}
Let $G$ and $H$ be graphs and let $u$ and $v$ be simplicial vertices in $G$ and $H$, respectively.
Then, vertex $(u,v)$ is simplicial in the strong product $G \boxtimes H$.
\end{lemma}

\begin{proof}
It suffices to show that the closed neighborhood $N_{G \boxtimes H}[(u,v)]$ is a clique in $G \boxtimes H$.
Note that $N_{G \boxtimes H}[(u,v)] = N_G[u]\times N_H[v]$,
the set $N_G[u]$ is a clique in $G$ (since $u$ is simplicial in $G$)
and, similarly, the set $N_H[v]$ is a clique in $H$.
The desired result now follows from the fact that the strong product of two complete graphs is a complete graph.
\end{proof}

Recall also that two distinct vertices $u$ and $v$ in a graph $G$ form a pair of \emph{true twins} if $N_{G}[u] = N_{G}[v]$.
We say that a graph is \emph{true-twin-free} if it contains no pair of true twins.
The next lemma shows that it suffices to characterize \po~strong product graphs in which both factors are true-twin-free.

\begin{lemma}\label{lemm:true-twins}
Let $G,G'$, and $H$ be graphs such that $G'$ is obtained from $G$ by adding a true twin.
Then, $G \boxtimes H$ is \po~if and only if $G' \boxtimes H$ is \po
\end{lemma}

\begin{proof}
Note that $G \boxtimes H$ is an induced subgraph of $G' \boxtimes H$. Therefore, by Proposition~\ref{prop:minors},
if $G' \boxtimes H$ is $1$-p.o., then so is $G \boxtimes H$.

Suppose now that $G \boxtimes H$ is $1$-p.o., and that $G'$ was obtained from $G$ by adding to it a true twin $x'$ to a vertex $x$ of $G$.
Note that for every $v\in V(H)$, we have $N_{G' \boxtimes H}[(x,v)]=N_{G'}[x] \times N_{H}[v]$ and $N_{G' \boxtimes H}[(x',v)]=N_{G'}[x'] \times N_{H}[v]$.
Since
$N_{G'}[x]=N_{G'}[x']$, each vertex of the form $(x',v)$ for $v \in V(H)$ is a true twin in $G' \boxtimes H$ of vertex $(x,v)$. It follows that $G'
\boxtimes H$ can be obtained from $G \boxtimes H$ by a sequence of true twin additions. By Proposition~\ref{prop:operations}, $G' \boxtimes
H$ is \po
\end{proof}

A similar approach as for the direct product (Lemma~\ref{lem:forbidden-direct})
gives the following necessary conditions for the strong product of two graphs to be \po

\begin{lemma}\label{lemm:forbiden}
Suppose that the strong product of two graphs $G$ and $H$ is $1$-p.o. Then:
\begin{enumerate}
\item If one of $G$ and $H$ contains an induced $P_3$, then the other one is $\{P_5, C_4, C_5, {\it claw}, {\it bull}\}$-free.
\item At least one of $G$ and $H$ is $P_4$-free.
\end{enumerate}
\end{lemma}

\begin{proof}
We can verify that each of the graphs $P_3 \boxtimes C_4$, $P_3 \boxtimes C_5$, $P_3 \boxtimes claw$, and $P_3 \boxtimes bull$ has
$K_{2,3}$ (the first element of family ${\cal F}_2$, see Fig.~\ref{fig:1}) as induced minor, that $P_3 \boxtimes P_5$
contains an induced copy of $F_2$, and that $P_4 \boxtimes P_4$ contains an induced copy of $F_1$.
Therefore, by Corollary~\ref{cor:minors}, none of these graphs is \po~We can observe such induced minors in Figure~\ref{fig:3}.

\begin{figure}[h!]
  \centering
\includegraphics[width=0.85\textwidth]{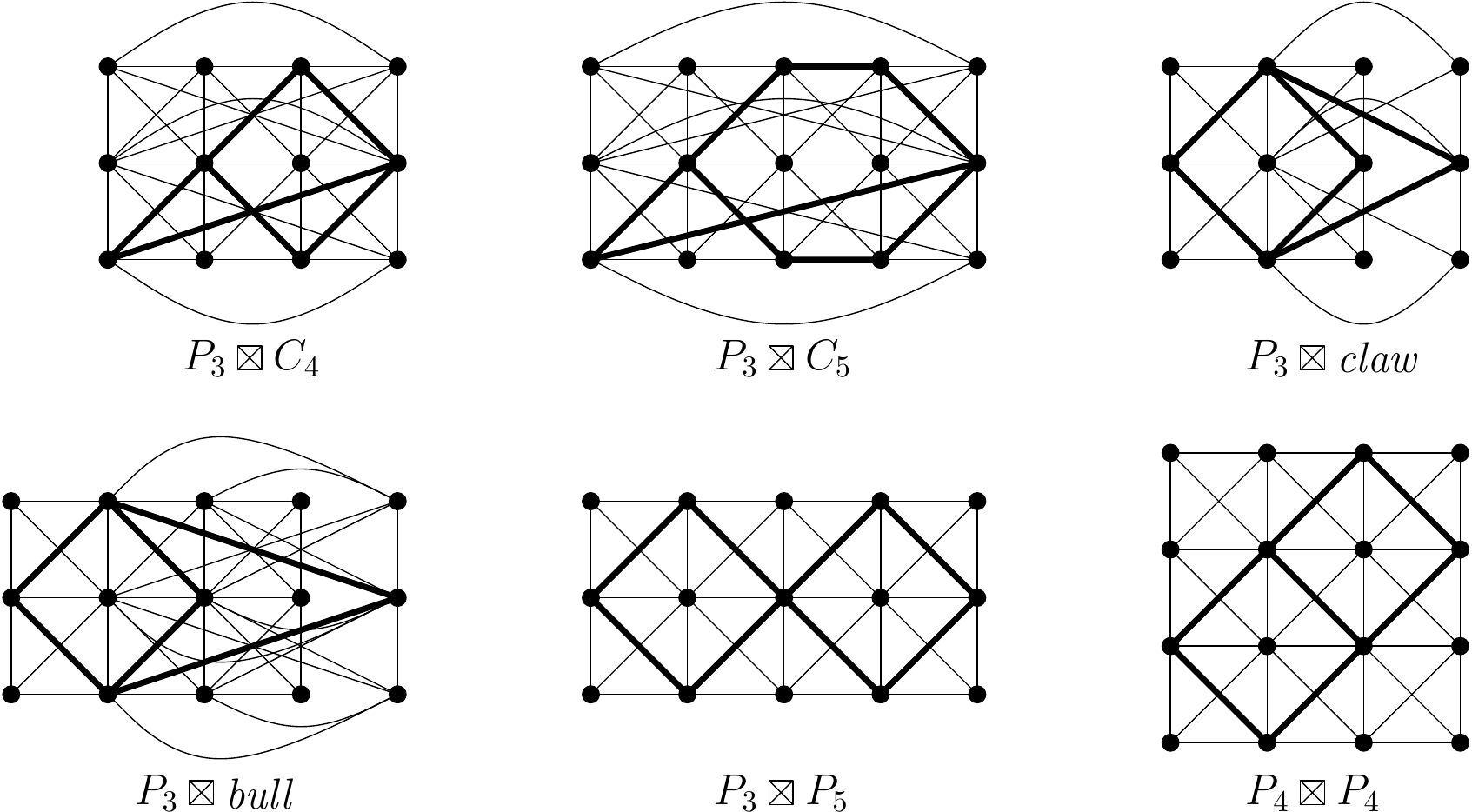}
\caption{$K_{2,3}$ as induced minor of $P_3 \boxtimes C_4$, $P_3 \boxtimes C_5$, $P_3 \boxtimes claw$, $P_3 \boxtimes bull$,
$F_2$ as induced subgraph of $P_3 \boxtimes P_5$, and
the domino ($F_1$) as induced subgraph of $P_4 \boxtimes P_4$.}\label{fig:3}
\end{figure}

The lemma now follows from the above observations and Corollary~\ref{cor:minors}.
\end{proof}

Lemma~\ref{lemm:forbiden} motivates the development of structural characterizations of $P_3$-free graphs, of $P_4$-free graphs, and of
$\{P_5, C_4, C_5, {\it claw}, {\it bull}\}$-free graphs. $P_3$-free graphs are precisely the disjoint union of complete graphs.
$P_4$-free graphs (also known as {\it cographs}) are also well understood: they are precisely the graphs that can be obtained from copies of $K_1$ by applying a sequence of the disjoint union and join operations~\cite{MR619603}.
The $\{P_5, C_4, C_5, {\it claw}, {\it bull}\}$-free graphs are characterized in the next section.

\subsection{The structure of $\{P_5, C_4, C_5, {\it claw}, {\it bull}\}$-free graphs}\label{subsec:co-chain}

Our characterization of $\{P_5, C_4, C_5, {\it claw}, {\it bull}\}$-free graphs will rely on the notion of co-chain graphs.
A~graph $G$ is a {\it co-chain graph} if its vertex set can be partitioned into two cliques, say $X$ and $Y$, such that
the vertices in $X$ can be ordered as $X = \{x_1,\ldots, x_{|X|}\}$ so that for all $1\le i<j\le |X|$, we have
$N[x_i]\subseteq N[x_j]$ (or, equivalently, $N(x_i)\cap Y\subseteq N(x_j)\cap Y$).
The pair $(X,Y)$ will be referred to as a {\it co-chain partition} of $G$.
The following observation is an immediate consequence of the definitions.

\begin{proposition}\label{prop:co-chain-TTA}
The set of co-chain graphs is closed under true twin additions and universal vertex additions.
\end{proposition}

The following structural characterization of connected $\{P_5, C_4, C_5, {\it claw}, {\it bull}\}$-free graphs can
also be seen as a forbidden induced subgraph characterization of co-chain graphs within connected graphs.

\begin{theorem}\label{prop:co-chain}
A connected graph $G$ is $\{P_5, C_4, C_5, {\it claw}, {\it bull}\}$-free if and only if it is co-chain.
\end{theorem}

\begin{proof}
Sufficiency of the condition is easy to establish. The graphs $P_5$, $C_5$, the claw, and the bull, are not co-bipartite and therefore not co-chain.
The $4$-cycle admits only one partition of its vertex set into two cliques, which however does not have the desired property.

Now we prove necessity. Let $G$ be a connected $\{P_5, C_4, C_5, {\it claw}, {\it bull}\}$-free graph.
We will show that $G$ is $3K_1$-free. This will imply that $G$ is co-chain
due to the known characterization of co-chain graphs as exactly the graphs that are
$\{3K_1, C_4, C_5\}$-free~\cite{MR2460558}.

Suppose for a contradiction that $G$ has an induced $3K_1$, with vertex set $\{x,y,z\}$, say.
Since $G$ is connected and $P_5$-free, every two vertices among $\{x,y,z\}$ are at distance $2$ or $3$.

Suppose first that $d(x,y) = d(x,z) = 2$. Let $y'$ be a common neighbor of
$x$ and $y$, and let $z'$ be a common neighbor of $x$ and $z$.
Since $G$ is claw-free, $y'z\not\in E(G)$ and similarly $yz' \not\in  E(G)$. In particular,
$y'\neq z'$.
Now, the vertex set $\{y,y',x,z',z\}$ induces either a $P_5$  (if $y'$ and $z'$ are non-adjacent), or a bull (otherwise), a contradiction.

Therefore, at least two out of the pairwise distances between $x$, $y$, and $z$ are equal to $3$.
By symmetry, we may assume that $d(x,y) = d(x,z) = 3$.
Note that the set of vertices at distance $2$ from $x$ form a clique, since otherwise
we could apply the arguments from the previous paragraph to the triple $\{x,y',z'\}$ where $\{y',z'\}$ is a
pair of non-adjacent vertices with $d(x,y') = d(x,z') = 2$.

Fix a pair of paths $P$ and $Q$ such that
$P= (x = p_0,p_1,p_2,p_3 = y)$  is a shortest $x$-$y$ path,
$Q= (x = q_0,q_1,q_2,q_3 = z)$  is a shortest $x$-$z$ path,
and $P$ and $Q$ agree in their initial segments as much as possible, that is,
the value of $k = k(P,Q) = \max\{j: p_i = q_i$ for all $0\le i\le j\}$ is maximized.
Clearly, $k\in \{0,1,2\}$.
If $k = 2$, then $G$ contains a claw induced by $\{p_1,p_2,y,z\}$. Therefore $k\in \{0,1\}$.
If $k = 1$, then, recalling that $p_2$ is adjacent to $q_2$, we infer that $G$ contains either
a claw induced by $\{p_1,p_2,y,z\}$ (if $p_2$ is adjacent to $z$) or a bull induced by $V(Q)\cup \{p_2\}$.
Therefore $k= 0$. By the minimality of $(P,Q)$, we infer that $\{p_1q_2, p_2q_1,p_2z,yq_2\}\cap E(G) = \emptyset$.
But now, $G$ contains a claw induced by $\{p_1,p_2,y,q_2\}$. This contradiction completes the proof.
\end{proof}

\subsection{Rafts and connected true-twin-free co-chain graphs}\label{subsec:rafts}

In Section~\ref{subsec:sufficient}, we will identify an infinite family of \po~strong product graphs.
The family will be based on the following particular family of co-chain graphs.
Given a non-negative integer $n\ge 0$, the \emph{raft of order $n$} is the graph
$R_n$ consisting of two disjoint cliques on $n+1$ vertices each, say $X = \{x_0, x_1, \ldots, x_n\}$ and
$Y = \{y_0, y_1, \ldots, y_n\}$ together with additional edges between $X$ and $Y$ such that for every $0
\leq i, j \leq n$, vertex $x_i$ is adjacent to vertex $y_j$ if and only if $i+j\ge n+1$.
Note that vertices $x_0$ and $y_0$ are simplicial in the raft.
The cliques $X$ and $Y$ will be referred to as the {\it parts} of the raft.
Fig.~\ref{fig:raft} shows rafts of order $n$ for $n\in \{1,2,3\}$.

\begin{figure}[h!]
  \centering
\includegraphics[width=0.8\textwidth]{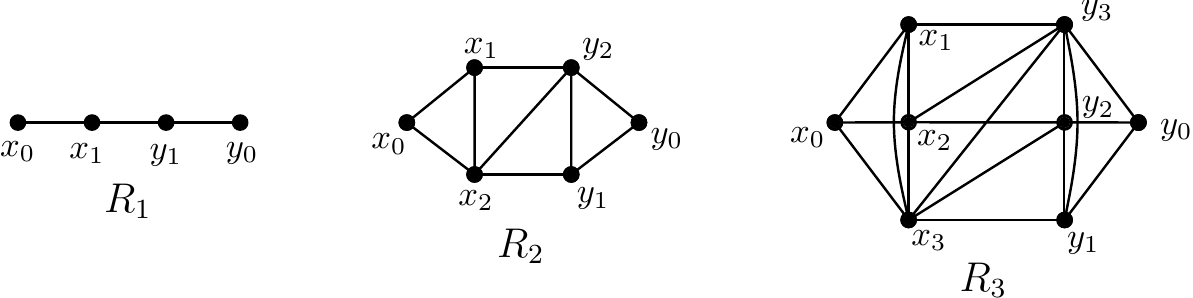}
\caption{Three small rafts}\label{fig:raft}
\end{figure}

It is an easy consequence of definitions that every raft is a co-chain graph.
Moreover, as we show next, rafts play a crucial role in the
classification of connected true-twin-free co-chain graphs.

\begin{proposition}\label{obs:p3}
Let $G$ be a connected true-twin-free graph.
Then, $G$ is co-chain if and only if $G \in \{K_1\}\cup \{R_n, n\geq 1\} \cup \{R_n\ast K_1, n \geq 0\}$.
Moreover, if $G$ is $P_4$-free, then $G$ is co-chain
if and only if $G\cong K_1$ or $G\cong P_3$.
\end{proposition}

\begin{proof}
Sufficiency is immediate since
every graph in $\{K_1\}\cup \{R_n, n\geq 1\} \cup \{R_n\ast K_1, n \geq 0\}$ is co-chain.

Now, let $G$ be a connected true-twin-free co-chain graph, with a co-chain partition $(X,Y)$.
Since $G$ is true-twin-free, the closed neighborhoods of vertices in $X = \{x_1,\ldots, x_{|X|}\}$ are properly nested.
Equivalently, $$N(x_1)\cap Y\subset N(x_2)\cap Y\subset \ldots \subset N(x_{|X|})\cap Y\,.$$
Since there are no pairs of true twins in $Y$, we have
$|N(x_{i+1})\cap Y| = |N(x_{i})\cap Y|+1$ for all $i \in \{1,\ldots, |X|-1\}$.
This implies an ordering of vertices in $Y$, say $Y = \{y_1,\ldots, y_{|Y|}\}$ such that
$N(y_{i})\cap X\subset N(y_{i+1})\cap X$ and
$|N(y_{i+1})\cap X| = |N(y_{i})\cap X|+1$
for all $i \in \{1,\ldots, |Y|-1\}$.

If $X = \emptyset$ or $Y = \emptyset$, then since both $X$ and $Y$ are cliques and $G$ is true-twin-free, we infer that $G\cong K_1$.

Now, both $X$ and $Y$ are non-empty, and we analyze four cases depending on the smallest neighborhoods of vertices in the two parts.
If $N(x_1)\cap Y = N(y_1)\cap X = \emptyset$, then since $G$ is connected, we have $|X|= |Y|\ge 2$,
and $G$ is isomorphic to $R_{|X|-1}$.
If $N(x_1)\cap Y = \emptyset$ and $N(y_1)\cap X \neq \emptyset$, then
$|X| \ge  2$, and deleting the universal vertex
$x_{|X|}$ from $G$ leaves a graph isomorphic to $R_{|X|-2}$.
Thus, $G\cong R_{|X|-2}\ast K_1$.
The case when $N(x_1)\cap Y \neq \emptyset$ and $N(y_1)\cap X = \emptyset$ is symmetric to the previous one.
Finally, if $N(x_1)\cap Y \neq \emptyset$ and $N(y_1)\cap X \neq  \emptyset$, then vertices
$x_{|X|}$ and $y_{|Y|}$ are both universal in $G$, contrary to the fact that $G$ is true-twin-free.

Suppose now that $G$ is also $P_4$-free but not isomorphic to either $K_1$ or $P_3$.
Note that since $R_1\cong P_4$, every raft of order at least $1$ contains an induced $P_4$.
It follows that $G$ is isomorphic to a graph of the form $R_n\ast K_1$ for some $n \geq 0$.
Since $R_0\ast K_1\cong P_3$, we have $n\ge 1$. But then $R_1\cong P_4$ is an induced subgraph of $G$, a contradiction.
\end{proof}

\subsection{An infinite family of \po~strong product graphs}\label{subsec:sufficient}

The following observation is an immediate consequence of Lemma~\ref{lem:simplicial}.

\begin{observation}\label{obs:strong}
Let $G$ be a graph with a simplicial vertex $v$, and let $P_3 = (u_1, u_2, u_3)$ be the $3$-vertex path, with
leaves $u_1$ and $u_3$. Then, vertices $(u_1, v)$ and $(u_3, v)$ are simplicial in $P_3 \boxtimes G$.
\end{observation}

\begin{proposition}\label{lemm:orient-rafts}
For every $n\ge 1$, the strong product $P_3 \boxtimes R_{n}$ is \po
\end{proposition}

\begin{proof}
First, notice that since $R_n$ has two simplicial vertices, Observation~\ref{obs:strong} implies that the product
$P_3 \boxtimes R_n$ has $4$ simplicial vertices.
Let $G$ be the product $P_3 \boxtimes R_n$ minus these $4$ simplicial vertices.
Since \po~graphs are closed under simplicial vertex additions,
it is enough to verify that $G$ is $1$-p.o.
To prove this we will give an explicit orientation of $G$ and show that it is a $1$-perfect orientation.

Let $V(P_3) = \{u_1, u_2, u_3\}$ where $u_1$ and $u_3$ are the two leaves.
Moreover, assuming the notation as in the definition of rafts, let $V(R_n)=X\cup Y$, where $X =  \{x_0, x_1, \ldots, x_n\}$ and $Y = \{y_0, y_1, \ldots, y_n\}$ are the
two parts of the raft. Vertices in $G$ will be said to be {\it left}, resp.~{\it right}, depending on whether their second coordinate is in $X$ or in $Y$, respectively.
A schematic representation of $G$ is shown in Fig.~\ref{fig:p3raft}. We partition the graph's vertex set into $8$ cliques:
two singletons, $\{a\}$ and $\{b\}$,
where $a = (u_2,x_0)$ and $b = (u_2,y_0)$,
and $6$ cliques of size $n$ each, namely
$A_{1}$, $A_{2}$, $A_{3}$, $B_{1}$, $B_{2}$, and $B_{3}$, defined as follows:
for $i\in \{1,2,3\}$, we have $B_i = \{u_i\}\times (X\setminus \{x_0\})$ and
$A_i = \{u_{4-i}\}\times (Y\setminus \{y_0\})$. Bold edges between certain pairs of sets
mean that every possible edge between the two sets is present.
If the corresponding edge is not bold, then only some of the edges between the two sets are present.

\begin{figure}[h!]
  \centering
\includegraphics[width=0.9\textwidth]{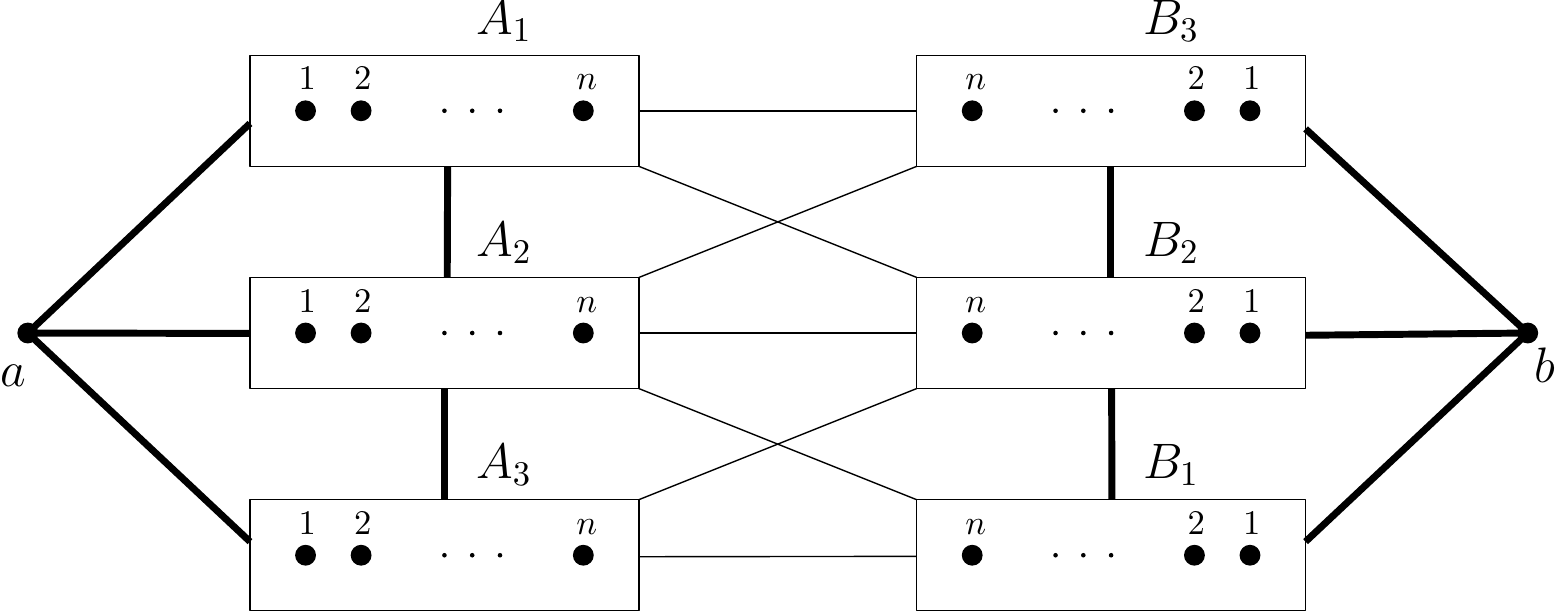}
\caption{A schematic representation of graph $G$}\label{fig:p3raft}
\end{figure}

To describe such edges, we introduce the following ordering of the vertices within each of the $6$ cliques of size $n$.
Note that for every $1\le i<j\le n$, we have that $N_{R_n}[x_i]\subset N_{R_n}[x_j]$ and $N_{R_n}[y_i]\subset N_{R_n}[y_j]$. We order the vertices in the $6$ cliques accordingly, that is,
for each clique of the form $A_i$, the linear ordering of its vertices
is $(u_i, x_1),\ldots, (u_i, x_n)$; for each clique of the form $B_i$, the linear ordering of its vertices
is $(u_{4-i}, y_1),\ldots, (u_{4-i}, y_n)$.
To keep the notation light, we will slightly abuse the notation, speaking of ``vertex $i$ in clique $C$''
(for $i\in \{1,\ldots, n\}$ and $C\in \{A_{1},A_2,A_3,B_1,B_2,B_3\}$) when referring to the $i$-th vertex
in the linear ordering of $C$. We will also speak of ``left'' and of ``right'' cliques.

The edges of graph $G$ can be now concisely described as follows.
We will say that two cliques $A_i$ and $A_j$ (or $B_i$ and $B_j$) are {\it adjacent} if $|i-j|\le 1$.
The neighborhood of $a$ is $A_1\cup A_2\cup A_3$.
The neighborhood of $b$ is $B_1\cup B_2\cup B_3$.
For each vertex $i$ in a left clique, say $A_j$, its closed neighborhood consists of vertex $a$,
all the vertices belonging to some left clique adjacent to $A_j$,
and of vertices $\{n-i+1, \ldots, n\}$ in each right clique adjacent to $B_{4-j}$.
For each vertex $i$ in a right clique, say $B_j$, its closed neighborhood consists of vertex $b$,
all vertices belonging to some right clique adjacent to $B_j$,
and of vertices $\{n-i+1, \ldots, n\}$ in each left clique adjacent to $A_{4-j}$.
Fig.~\ref{fig:p3raft2} shows a concrete example of $G$, namely for the case $n = 3$.

\begin{figure}[h!]
  \centering
\includegraphics[width=0.7\textwidth]{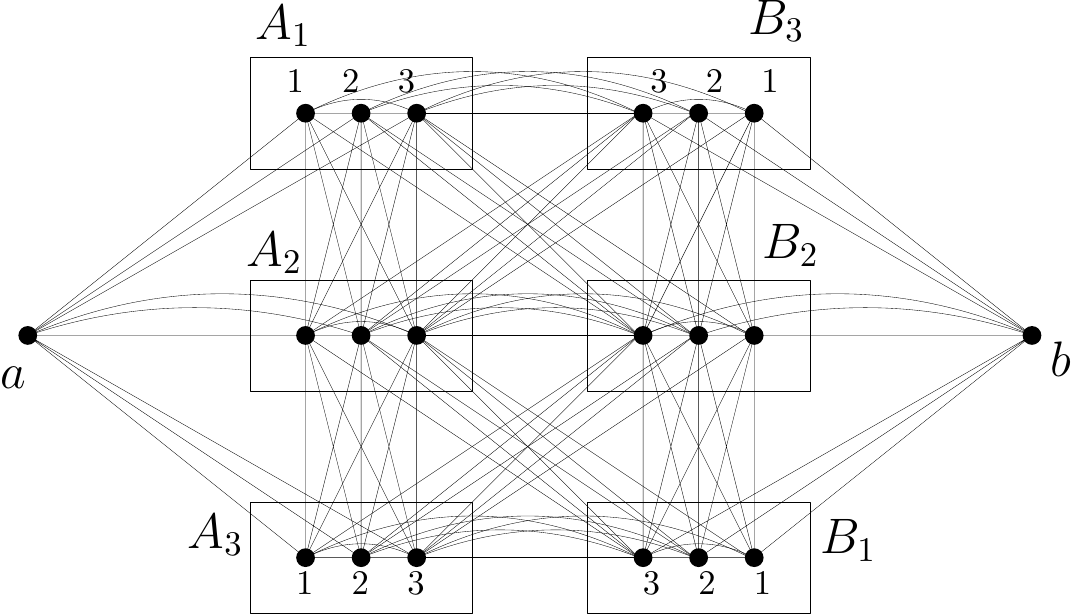}
\caption{Graph $G$ in the case $n = 3$}\label{fig:p3raft2}
\end{figure}

We now define an orientation of $G$, say $D$, as follows:
\begin{enumerate}[--]
\item Edges between vertex $a$ and a vertex $i\in A_j$ are oriented from $i$ to $a$ for $j= 1$ and
from $a$ to $i$ for $j\in \{2,3\}$.
Symmetrically, edges between vertex $b$ and a vertex $i \in B_j$ are oriented from $i$ to $b$ for $j= 1$ and
from $b$ to $i$ for $j\in \{2,3\}$.

\item Edges within each clique are oriented from vertex $i$ to vertex $j$
(with $j\neq i$)  if and only if $i < j$.

\item All edges between vertices in $A_1$ and $A_2$ are oriented from $A_1$ to $A_2$.
Symmetrically, all edges between vertices in $B_1$ and $B_2$ are oriented from $B_1$ to $B_2$.

\item Edges between vertices in $A_2$ and $A_3$ are oriented as follows:
For $i \in A_2$ and $j \in A_3$, from $i$ to $j$ if $i<j$, and from $j$ to $i$, otherwise.
Symmetrically, edges between $B_2$ and $B_3$ are oriented as follows:
For $i \in B_2$ and $j \in B_3$, from $i$ to $j$ if $i<j$, and from $j$ to $i$, otherwise.

\item All edges between vertices in $A_1$ and $B_3$ are oriented from $B_3$ to $A_1$.
Symmetrically, all edges between vertices in $A_3$ and $B_1$ are oriented from $A_3$ to $B_1$.

\item All edges between vertices in $A_1$ and $B_2$ are oriented from $B_2$ to $A_1$.
Symmetrically, all edges between vertices in $A_2$ and $B_1$ are oriented from $A_2$ to $B_1$.

\item All edges between vertices in $A_2$ and $B_1$ are oriented from $B_1$ to $A_2$.
Symmetrically, all edges between vertices in $A_3$ and $B_2$ are oriented from $A_3$ to $B_2$.

\item Finally, all edges between vertices in $A_2$ and $B_2$ are oriented from $A_2$ to $B_2$.
\end{enumerate}

To conclude the proof  it remains to check that $D$ is a $1$-perfect orientation of $G$, that is,
that for each vertex $v$ in $G$, its out-neighborhood in $D$ forms a clique in $G$.
We consider several cases according to which part of the above vertex partition vertex $v$ belongs to:
\begin{enumerate}[(i)]
\item $v \in \{a,b\}$. We have $N^{+}_D(a)=A_2 \cup A_3$, which forms a clique in $G$.
Symmetrically, $N^{+}_D(b)$ forms a clique in $G$.

\item $v \in A_1\cup B_1$.
By symmetry, we may assume that $v\in A_1$, say $v = i$. Then, $N^{+}_D(i)= \{a\} \cup \{j \in A_1, j > i\}\cup A_2$, which forms a clique in $G$.

\item $v \in A_2$, say $v = i$. We have $N^{+}_D(i)= A\cup B$, where $A =\{j \in A_2\cup A_3, j > i\}$ and $B = \{j \in B_2\cup B_3, j > n-i\}$.
Note that $A$ and $B$ are cliques in $G$.
Moreover, if $j\in A$ and $k\in B$, then $j+k>i+(n-i) = n$, which implies that $j$ and $k$ are adjacent in $G$.
Therefore, $N^{+}_D(i)$ is a clique in $G$.

\item $v \in A_3\cup B_3$.
By symmetry, we may assume that $v\in A_3$, say $v = i$. We have $N^{+}_D(i)= A\cup B$, where $A =\{j \in A_2, j\ge i\}\cup \{j\in A_3, j > i\}$
and $B = \{j \in B_2\cup B_3, j > n-i\}$.
Again, $A$ and $B$ are cliques in $G$.
Moreover, if $j\in A$ and $k\in B$, then $j+k\ge i+(n-i+1) = n+1$, which implies that $j$ and $k$ are adjacent in $G$.
Therefore, $N^{+}_D(i)$ is a clique in $G$.

\item $v \in B_2$, say $v = i$.
Then, $N^{+}_D(i)= A \cup B$, where $A = \{j \in A_1, j > n-i\}$ and
$B = \{j \in B_2\cup B_3, j > i\}$.
Again, $A$ and $B$ are cliques in $G$ such that if $j\in A$ and $k\in B$,
then $j+k>(n-i)+i = n$, which implies that $j$ and $k$ are adjacent in $G$.
Therefore, $N^{+}_D(i)$ is a clique in $G$.
\end{enumerate}

This completes the proof that $G$ is \po
\end{proof}

Note that for every $n\ge 0$, the graph $R_n\ast K_1$ is isomorphic to an induced subgraph of $R_{n+2}$.
Therefore, Proposition~\ref{lemm:orient-rafts} and the fact that \po~graphs are closed under taking induced subgraphs implies the following.

\begin{corollary}\label{cor:orient-rafts}
For every $n\ge 0$, the strong product $P_3 \boxtimes (R_{n}\ast K_1)$ is \po
\end{corollary}

\subsection{A characterization of nontrivial strong product graphs that are $1$-p.o.}\label{subsec:characterization}

We now derive the main result of this section. We first show that Proposition~\ref{lemm:orient-rafts} and Corollary~\ref{cor:orient-rafts}
describe all nontrivial strong products of two true-twin-free connected graphs that are $1$-p.o.

\begin{lemma}\label{thm:strong}
A nontrivial strong product, $G \boxtimes H$, of two true-twin-free connected graphs $G$ and $H$ is \po~if and only if one of them is isomorphic to
$P_3$ and the other one belongs to $\{R_{n}, n \geq 1\}\cup \{R_{n} \ast K_1, n\geq 0\}$.
\end{lemma}

\begin{proof}
If $G \cong P_3$ and $H \in \{R_n, n \geq 1\} \cup \{R_n \ast K_1,  n \ge 0\}$, the strong product $G \boxtimes
H$ is \po~by Proposition~\ref{lemm:orient-rafts} and Corollary~\ref{cor:orient-rafts}.

Conversely, suppose that $G \boxtimes H$ is \po~If one of the factors is $P_3$-free, its connectedness would imply that the graph is complete and
therefore contains a pair of true twins, which is a contradiction. Thus, both factors contain an induced $P_3$, and
by the first part of Lemma~\ref{lemm:forbiden}, they are both $\{P_5, C_4, C_5, {\it claw}, {\it bull}\}$-free.
In particular, by Theorem~\ref{prop:co-chain}, they are both co-chain.
Moreover, by Proposition~\ref{obs:p3}, they both belong to the set $\{R_n, n\geq 1\} \cup \{R_n\ast K_1, n \geq 0\}$.
By the second part of Lemma~\ref{lemm:forbiden}, at least one of $G$ and $H$ is $P_4$-free, and thus, by
Proposition~\ref{obs:p3}, isomorphic to $P_3$.
\end{proof}


To describe the main result of this section, the following notions will be convenient.
We say that a graph is {\it $2$-complete} if it is the union of two (not necessarily distinct) complete graphs sharing
at least one vertex. Equivalently, a graph is $2$-complete if and only if it can be obtained from either $K_1$ or $P_3$ by applying a sequence of true twin additions.
Moreover, a {\it true-twin-reduction} of a graph $G$ is any maximal induced subgraph of $G$ that is true-twin-free. It is easy to observe that any two true-twin-reductions of a graph $G$ are isomorphic to each other, thus we can speak of
{\it the} true-twin-reduction of $G$.

\begin{theorem}\label{thm:strong-general}
A nontrivial strong product, $G\boxtimes H$, of two graphs $G$ and $H$ is \po~if and only if one of the following conditions holds:
\begin{enumerate}[(i)]
\item Every component of $G$ is complete and $H$ is \po, or vice versa.
\item Every component of $G$ is $2$-complete and
every component of $H$ is co-chain, or vice versa.
\end{enumerate}
\end{theorem}

\begin{proof}
Suppose first that given two nontrivial graphs $G$ and $H$, the product $G\boxtimes H$ is $1$-p.o.
Then, $G$ and $H$ are both $1$-p.o.
We may assume that not every component of $G$ is complete and not
every component of $H$ is complete (since otherwise the first condition holds).
Let $G'$ and $H'$ be the true-twin-free reductions of $G$ and $H$, respectively.
Clearly, $G'$ and $H'$ are true-twin-free and $G'\boxtimes H'$ is \po~(since it is an induced subgraph of $G\boxtimes H$).
Let  $G'_1, \ldots, G'_k$ be the components of $G'$ and let
$H'_1, \ldots, H'_\ell$, be the components of $H'$. Then, the components of $G'\boxtimes H'$ are
of exactly $G'_i \boxtimes H'_j$ for $1\leq i\leq k$, $1\leq j\leq \ell$.
Since not every component of $G$ is complete, $G'$ has a nontrivial component, say $G_1'$,
and, similarly, $H'$ has a nontrivial component, say $H_1'$.
Let $i\in \{1,\ldots, k\}$ be such that $G_i'$ is a nontrivial component of $G'$.
Applying Lemma~\ref{thm:strong} to the product $G_i'\boxtimes H_1'$ (which is $1$-p.o.),
we infer that $G_i'$ belongs to the set $\{R_n, n\geq 1\} \cup \{R_n \ast K_1, n\geq 0\}$.
Similarly, every nontrivial component of $H'$ belongs to the set $\{R_n, n\geq 1\} \cup \{R_n \ast K_1, n\geq 0\}$.
In particular, every component of $G'$ or of $H'$ is co-chain.
By Proposition~\ref{lemm:forbiden}, at least one of $G$ and $H$ is $P_4$-free, which, since
$P_4\cong R_1$,  implies that not both $G'$ and $H'$ contain an induced $R_n$ (for some $n\geq 1$).
Therefore, we may assume that each component of $G'$ is isomorphic to either $K_1$ or to $R_0\ast K_1\cong P_3$.
Hence, every component of $G$ is $2$-complete. Since every component of $H$ is obtained from a component of $H'$ by a sequence of true twin additions and
every component of $H'$ is co-chain, Propositions~\ref{prop:co-chain-TTA} and~\ref{obs:p3} imply that
every component of $H$ is co-chain, as claimed.

Let us now prove that each of the two conditions is also sufficient. Suppose first that every component of $G$ is complete and $H$ is $1$-p.o.
Then, every component of $G\boxtimes H$ is isomorphic to the strong product of a complete graph with a \po~graph, say $H'$,
and can thus be obtained by applying a sequence of true twin additions to vertices of $H'$.
Applying Proposition~\ref{prop:operations} and Corollary~\ref{cor:connected}, we infer that
$G \boxtimes H$ is $1$-p.o.~in this case.
In the other case, every component of $G$ is $2$-complete and every component of $H$ is co-chain.
By Corollary~\ref{cor:connected}, it suffices to consider the case when $G$ and $H$ are connected, and, moreover,
we may assume by Lemma~\ref{lemm:true-twins} that they are both true-twin-free.
In particular, $G$ is isomorphic to one of $K_1$ and $P_3$, and
by Proposition~\ref{obs:p3}, $H$ is isomorphic to a graph from
the set $\{K_1\}\cup \{R_n, n\geq 1\} \cup \{R_n \ast K_1, n\geq 0\}$.
The fact that $G\boxtimes H$ is \po~now follows from
Proposition~\ref{lemm:orient-rafts}, Corollary~\ref{cor:orient-rafts},
and the fact that \po~graphs are closed under taking induced subgraphs.
\end{proof}

\def\soft#1{\leavevmode\setbox0=\hbox{h}\dimen7=\ht0\advance \dimen7
  by-1ex\relax\if t#1\relax\rlap{\raise.6\dimen7
  \hbox{\kern.3ex\char'47}}#1\relax\else\if T#1\relax
  \rlap{\raise.5\dimen7\hbox{\kern1.3ex\char'47}}#1\relax \else\if
  d#1\relax\rlap{\raise.5\dimen7\hbox{\kern.9ex \char'47}}#1\relax\else\if
  D#1\relax\rlap{\raise.5\dimen7 \hbox{\kern1.4ex\char'47}}#1\relax\else\if
  l#1\relax \rlap{\raise.5\dimen7\hbox{\kern.4ex\char'47}}#1\relax \else\if
  L#1\relax\rlap{\raise.5\dimen7\hbox{\kern.7ex
  \char'47}}#1\relax\else\message{accent \string\soft \space #1 not
  defined!}#1\relax\fi\fi\fi\fi\fi\fi} \def\cprime{$'$}

\end{document}